\documentclass[12pt]{article}

\usepackage{amsmath,amssymb,  latexsym, amsthm, amscd, fancyhdr, fullpage, url}

\usepackage[svgnames]{xcolor}
\usepackage{tikz}

\pgfdeclarelayer{edgelayer}
\pgfdeclarelayer{nodelayer}
\pgfsetlayers{edgelayer,nodelayer,main}

\tikzstyle{none}=[inner sep=0pt]
\definecolor{hexcolor0xf81e1c}{rgb}{0.973,0.118,0.110}
\definecolor{hexcolor0x3c00ff}{rgb}{0.235,0.000,1.000}
\definecolor{hexcolor0x24fe00}{rgb}{0.141,0.996,0.000}

\tikzstyle{whitevertex}=[circle,fill=White,draw=Black, scale = 0.5]
\tikzstyle{vertex}=[circle,fill=White,draw=Black, scale = 0.5]
\tikzstyle{redvertex}=[circle,fill=hexcolor0xf81e1c,draw=Black]
\tikzstyle{bluevertex}=[circle,fill=hexcolor0x3c00ff,draw=Black]
\tikzstyle{greenvertex}=[circle,fill=hexcolor0x24fe00,draw=Black]
\tikzstyle{textbox}=[rectangle,fill=none,draw=none]

\tikzstyle{arc}=[Black, ->]

\newtheorem{theorem}{Theorem}[section]

\newtheorem{corollary}[theorem]{Corollary}
\newtheorem{lemma}[theorem]{Lemma}
\newtheorem{proposition}[theorem]{Proposition}
\newtheorem{observation}[theorem]{Observation}

\begin{document}

\title{Switching $(m, n)$-mixed graphs with respect to Abelian groups}

\author{Etienne Leclerc\thanks{Mathematics and Statistics, University of Victoria, CANADA, etiennel@uvic.ca, gmacgill@uvic.ca}\,, 
Gary MacGillivray\thanks{Research supported by NSERC}\ $^*$,
Jacqueline M. Warren
\thanks{Department of Mathematics, University of California, San Diego, USA, jmariadwarren@gmail.com}}
\date{}

\maketitle 

\begin{abstract}
We extend results of Brewster and Graves for switching $m$-edge coloured graphs 
with respect to a cyclic group to switching $(m, n)$-mixed graphs with respect to an
Abelian group.  
In particular, we establish the existence of a $(m, n)$-mixed graph $P_\Gamma(H)$
with the property that a $(m, n)$-mixed graph $G$ is switch equivalent to $H$ if and only if it is a special 
subgraph of $P_\Gamma(H)$, and the property that that $G$ can be switched to have a homomorphism to  $H$ if and only if it has a homomorphism (without switching) to  $P_\Gamma(H)$.
We consider the question of deciding whether a $(m, n)$-mixed graph can be switched so that it 
has a homomorphism to a proper subgraph, i.e. whether it can be switched so 
that it isn't a core.  We show that this question is NP-hard for arbitrary
groups and NP-complete for Abelian groups.
Finally, we consider the complexity of the switchable $k$-colouring problem
for $(m, n)$-mixed graphs and prove a dichotomy theorem in the cases where $m \geq 1$.
\end{abstract}

\section{Overview}

In what follows, we define a switching operation for $(m, n)$-mixed graphs in such a way
that it generalizes the pushing operation for oriented graphs (e.g. see \cite{K, KM}), and the switching operation for $m$-edge-coloured graphs (e.g. see \cite{BG, SS}).
The first step is to define $(m, n)$-mixed graphs and indicate how they generalize oriented 
graphs and edge-coloured graphs.
That is done in the next section.
In Section 3 we define the general switching operation with respect to a finite group, and
 the concept of switch equivalence.
Section 4 is concerned with the concept of a switchable homomorphism.
We prove results which show that  the main results of the theory previously developed for 
switching hold in this more general context.
In Section 5 we focus on the situation where the group of possible switches is Abelian,
and show that other results from the theories of  pushing and switching hold.  
In particular, for each $(m, n)$-mixed graph $H$ there is a special $(m, n)$-mixed graph $P(H)$ such that $G$ can be switched so that the resulting graph has a homomorphism to $H$ if and only if there is a homomorphism of $G$ to $P(H)$.
This theorem extends and contains results from \cite{BG, KM, SS}.
Section 6 is concerned with the concept of a switchable core.
The main results include an extension of  a result of Brewster and Graves \cite{BG} for switching  $m$-edge-coloured graphs with respect to a cyclic group  to  switching $(m, n)$-mixed graphs with respect to an Abelian group,
and determining the complexity of the problem of deciding whether is a given $(m, n)$-mixed graph
is not a core.
In the last section we consider the complexity of the $k$-colouring problem for $(m, n)$-mixed graphs.
We prove a dichotomy theorem in the cases where $m \geq 1$.

\section{Homomorphisms of $(m, n)$-mixed graphs}

A \textit{mixed graph} is an ordered triple $G=(V(G), E(G), A(G))$, where 
$V(G)$ is a set of objects called \emph{vertices},
$E(G)$ is a set of unordered pairs of not necessarily distinct vertices called \emph{edges}, and 
$A(G)$ is a set of ordered pairs of not necessarily distinct vertices called \emph{arcs}.

An \emph{$(m, n)$-mixed graph} is a mixed graph in which each edge is assigned one of the 
$m$ colours $1, 2, \ldots, m$ and each arc is assigned one of the $n$ colours $1, 2, \ldots, n$.
For $1 \leq i \leq m$, let $E_i(G)$ be \emph{the set of edges of colour $i$}, and
for $1 \leq j \leq n$, let  $A_j(G)$ be  \emph{the set of arcs of colour $j$}.

By definition, a mixed graph $G$ is obtained from a simple undirected graph $\mathit{underlying}(G)$
by choosing an orientation for some of its edges.
A $(m,n)$-mixed graph $H$ arises from a simple undirected graph $\mathit{underlying}(H)$
by choosing an orientation for some of its edges, and then a colour for each edge and each arc.
It is possible to make a more general definition in which loops and multiple edges are 
allowed, but we will not do so.  
We consider only $(m,n)$-mixed graphs which arise from simple undirected graphs in the way described above.

When the context is clear we write $V, E$ and $A$ instead of $V(G), E(G)$ and $A(G)$, respectively, and 
similarly for other subsets or parameters related to a graph.

Let $G$ and $H$ be $(m, n)$-mixed graphs.
A \emph{homomorphism} of $G$ to $H$ is a function $f: V(G) \to V(H)$ such that
if the edge $ab\in E_i(G)$, then $f(a)f(b)\in E_i(H)$,
and if the arc $ab\in A_j(G)$ then $f(a)f(b)\in A_j(H)$. 
If there is a homomorphism of $G$ to $H$, we may write $G\rightarrow H$.  
A homomorphism $G \to H$ preserves edges, arcs, and colours.

Observe that $(1, 0)$-mixed graphs are undirected graphs, $(m, 0)$-mixed graphs are $m$-edge-coloured graphs, and $(0, 1)$-mixed graphs are directed graphs.
The definition of homomorphism given above restricts to the usual definition in each of these cases.
The $(m, n)$-mixed graphs and homomorphisms between them were first introduced in 
\cite{NR} as a means of unifying the seemingly parallel theories of homomorphisms of 
oriented graphs and homomorphisms of 2-edge-coloured graphs: quite often statements that
hold for oriented graphs, for example, also hold for 2-edge-coloured graphs with virtually the same proof (e.g. see \cite{KSZ}).  In such cases it is reasonable to look for a general theorem about $(m, n)$-mixed graphs that contains these statements as special cases (for example, see \cite{NR}).

\section{Switching and switch equivalence}

We now define the switching operation on a $(m, n)$-mixed graph $G$. 
Let $\phi\in S_m$, $\psi\in S_n$, and $\pi=(p_1,p_2, \ldots, p_n)\in(\mathbb{Z}_2)^n$.
Informally, for the ordered triple $\gamma=(\phi,\psi,\pi)$ and $v \in V$, define \emph{switching at $v$ with respect to $\gamma$} to be the operation that transforms $G$ into the $(m, n)$-mixed graph $G^{(v, \gamma)}$ by 
permuting  colours of the edges incident with $v$ according to $\phi$, the colours 
of the arcs incident with $v$ according to $\psi$, and 
reversing the orientation of arcs of colour $j$ incident with $v$ if and only if $p_j=1$. 
Arc reversals are imagined as happening at the same time as colour changes.
Edges can only switch colours with edges, and arcs can only switch colours with arcs.
If $vv$ is a loop (edge or arc), then the permutation $\phi$ or $\psi$ is applied twice because 
there are two incidences with $v$.
Thus,  edges of colour $i$ joining distinct vertices in $G$ have colour $\phi(i)$ in $G^{(v, \gamma)}$,
arcs of colour $j$ joining distinct vertices in $G$ have colour $\psi(j)$ in $G^{(v, \gamma)}$ and 
their orientation is reversed if and only if $p_j = 1$.
Edge loops in $G$ of colour $i$ have colour $\phi(\phi(i))$ in $G^{(v, \gamma)}$, and
arc loops in $G$ of colour $i$ have colour $\psi(\psi(i))$ in $G^{(v, \gamma)}$

It follows from the definition of the wreath product $S_2\wr S_n$ (also known as the hyperoctahedral group) 
that the collection of all possible switches that can be applied at vertices of a $(m, n)$-mixed graph is naturally isomorphic to the direct product pf $S_m$ with the wreath product 
$S_m\otimes(S_2\wr S_n)$. 
As an aside, it is known that the hyperoctahedral group is the automorphism group of the $n$-dimensional hypercube \cite{H}.


If the permutation group $\Gamma$ is a  subgroup of $S_m\otimes(S_2\wr S_n)$, then we will call $\Gamma$ a \emph{$(m, n)$-switching group}.

Let $G$ be a $(m, n)$-mixed graph and $\Gamma$ be a $(m, n)$-switching group.
A \emph{switching pair} is an ordered pair $(v, \gamma)$, where $v \in V$ and $\gamma \in \Gamma$,
Given a sequence of switching pairs $\mathcal{S} = (v_1, \gamma_1), (v_2, $ $\gamma_2), \ldots, (v_t, \gamma_t)$, where $v_i \in V$ and $\gamma_i \in \Gamma$ for $i = 1, 2, \ldots, t$, we recursively define
$$G^{\mathcal{S}} = G^{(v_1, \gamma_1), (v_2, \gamma_2), \ldots, (v_t, \gamma_t)}
= \left(G^{(v_1, \gamma_1)}\right)^{(v_2, \gamma_2), (v_3, \gamma_3), \ldots, (v_t, \gamma_t)}.$$

\begin{proposition}
Let $X = \{v_1, v_2, \ldots, v_t\}$ be an independent set in the $(m, n)$-mixed graph $G$, and
let $\Gamma$ be a $(m, n)$-switching group.
Then for any permutation $\pi$ of the elements of $X$ and any $\gamma \in \Gamma$,
$$G^{(v_1, \gamma), (v_2, \gamma), \ldots, (v_t, \gamma)}
= G^{(\pi(v_1), \gamma), (\pi(v_2), \gamma), \ldots, (\pi(v_t), \gamma)}.$$
\label{SwitchIndep}
\end{proposition}
{\bf Proof.}
Since $X$ is an independent set, the colour of each edge or arc is affected by at most one switch.
Therefore the order in which the switches are applies does not matter.
$\Box$

Let $\Gamma$ be a $(m, n)$-switching group.
The relation $\sim_\Gamma$ on the set of all $(m, n)$-mixed graphs defined by 
$G \sim_\Gamma H$ if and only if there is a finite sequence $\mathcal{S}$ of switching pairs
such that $G^{\mathcal{S}} = H$ is an equivalence relation.  
When $G \sim_\Gamma H$ we say that $G$ and $H$ are \emph{$\Gamma$-switch equivalent}.
The \emph{$\Gamma$-switching equivalence class of the $(m, n)$-mixed graph $G$} is denoted by $[G]_\Gamma$.

\section{Switchable Homomorphisms}

Let $G$ and $H$ be $(m, n)$-mixed graphs.
We say that there is a  \textit{$\Gamma$-switchable homomorphism} of $G$ to $H$ if there is
a homomorphism from some element $G'\in[G]_\Gamma$ to $H$.
If some such homomorphism $f$ exists, we say that \emph{$G$ is $\Gamma$-switchably homomorphic} to $H$, and write $f: G\xrightarrow[\Gamma]{} H$.
We write $G\xrightarrow[\Gamma]{} H$ when the existence of a homomorphism matters, and its name does not.
 
Our definition for $\Gamma$-switching restricts to the definitions given in \cite{BG, KM, SS} for switching or pushing in the case of $(m,0)$ or $(0,1)$ graphs.
In the former case, $\Gamma$ consists of elements of the form $(\phi,e,e)$, where $e$ is the 
identity element of the relevant group.
In the latter case, $\Gamma$ consists of the two elements $(e,e,0)$ and $(e,e,1)$. 
Our definition of a $\Gamma$-switchable homomorphism thus naturally restricts to the previous definitions of switching and pushing.

\begin{observation}
Let $G$ and $H$ be $(m, n)$-mixed graphs and let $\Gamma$ be a $(m, n)$-switching group.
A $\Gamma$ switchable homomorphism of $G$ to $H$ is a homomorphism of 
$\mathit{underlying}(G)$ to $\mathit{underlying}(H)$.
\label{UnderlyingHom}
\end{observation}

We will show that $\Gamma$-switchable homomorphisms compose.  Some easy consequences
of the definitions above are recorded first.

\begin{observation}\label{Obs3.1}
Let $G$ and $H$ be  $(m, n)$-mixed graphs, and let $\Gamma$ be a $(m, n)$-switching group.  Then
\begin{enumerate}
\item if $G  \in [H]_\Gamma$, then $G \xrightarrow[\Gamma]{} H$;
\item if $G \to H$, then $G \xrightarrow[\Gamma]{} H$; 
\item if $F$ is a subgraph of $G$ and $G \xrightarrow[\Gamma]{} H$, then $F \xrightarrow[\Gamma]{} H$.
\end{enumerate}
\end{observation}

\begin{proposition}
Let $F, G$ and $H$ be $(m, n)$-mixed graphs, and let  $\Gamma$ be a $(m, n)$-switching group.
If  $G \xrightarrow[\Gamma]{} H$ and $H \xrightarrow[\Gamma]{} F$,
then $G \xrightarrow[\Gamma]{} F$.
\label{compose}
\end{proposition}

\begin{proof}
Let $G^\prime \in [G]_\Gamma$ be such that $g: G^\prime \to H$ is a homomorphism,
and $H^\prime \in [H]_\Gamma$ be such that $h: H^\prime \to F$ is a homomorphism.
We show that there exists $G^{\prime\prime} \in [G]_\Gamma$ such that $(h \circ g): G^{\prime\prime} \to F$.

By hypothesis, there is a sequence of switching pairs
$\mathcal{S} = (v_1, \gamma_1), (v_2, $ $\gamma_2), \ldots, (v_t, \gamma_t)$
such that $H^{\mathcal{S}} = H^\prime$.
Let $G^{\prime\prime}$ be the $(m, n)$-mixed graph arising from applying 
$\gamma_{1}$ to all elements of $g^{-1}(v_1)$ in $V(G^\prime)$, then 
$\gamma_{2}$ to all elements of $g^{-1}(x_2)$ in $V(G^\prime)$,
and so on.  (Since $g$ is a homomorphism, each set  $g^{-1}(v_i)$ is independent in $G^\prime$. Hence by
Proposition \ref{SwitchIndep} the order in which the switches are applied to the elements of this set does not matter.)

We claim that $(h \circ g): G^{\prime\prime} \to H^\prime$ is a homomorphism.
Suppose $xy \in E_i(G^\prime)$ and $xy \in E_j(G^{\prime\prime})$.  
We know $g(x)g(y) \in E_i(H)$.
The edge $xy$ is acted on by exactly the same permutations, in exactly the same order, 
in the formation of $G^{\prime\prime}$ from $G^\prime$ as is the edge $g(x)g(y)$ in the formation of $H^\prime$ from $H$.
Therefore $g(x)g(y) \in E_j(H^\prime)$ and, since $h$ is a homomorphism, $h(g(x))h(g(y)) \in E_j(F)$.
The argument is similar when $xy \in A_i(G')$.
Thus, $h \circ g$ is a homomorphism of $G^{\prime\prime}$ to $F$. 
\end{proof}

We conclude this section by establishing some other useful properties of $\Gamma$-switchable homomorphisms.

\begin{theorem} 
Let $G$ and $H$ be  $(m, n)$-mixed graphs, and let $\Gamma$ be a $(m, n)$-switching group.  
If $G \to H$, then
\begin{enumerate}
\item  $G^\prime \xrightarrow[\Gamma]{} H$ for any $G^\prime \in [G]_\Gamma$; 
\item for any $H^\prime \in [H]_\Gamma$ there exists $G' \in [G]_\Gamma$ such that $G' \to H'$. 
\end{enumerate}
\label{SwitchToHom}
\end{theorem}

\begin{proof}
Statement 1 is immediate from the definition of a $\Gamma$-switchable homomorphism, Observation \ref{Obs3.1} and Proposition \ref{compose}.

We now prove statement 2.
Let $h: G \to H$ be a homomorphism.
Since $H^\prime \in [H]_\Gamma$, there exists a sequence of switching pairs
$\mathcal{S}= (v_1, \gamma_1), (v_2, $ $\gamma_2), \ldots, (v_t, \gamma_t)$
such that $H^{\mathcal{S}} = H^\prime$.
The result now follows (as in Proposition \ref{compose}) by
letting $G^\prime$ be the $(m, n)$-mixed graph arising from applying 
$\gamma_{1}$ to all elements of $h^{-1}(x_1)$, then 
$\gamma_{2}$ to all elements of $h^{-1}(x_2)$,
and so on.  The mapping $h: G^\prime \to H^\prime$ is a homomorphism.
\end{proof}

\begin{corollary}
Let $G$ and $H$ be  $(m, n)$-mixed graphs, and let $\Gamma$ be a $(m, n)$-switching group. 
If $G \xrightarrow[\Gamma]{} H$, then for all  $G' \in [G]$ and $H' \in [H]$ we have $G' \xrightarrow[\Gamma]{} H'$.
\end{corollary}

Notice that the statement: \emph{``if $G \to H$ is a homomorphism and $G' \in [G]_\Gamma$, then there exists $H' \in [H]_\Gamma$ such that $G' \to H'$.''} is false. 
Consider the 2-edge coloured graph $G$ consisting of four vertices and two edges, so that each vertex has degree 1, and each edge is in $E_0$. 
Then, $G \to H$, where $H$ consists of two vertices joined by an edge in $E_0$.   
Let $\Gamma = S_2$, and $G^\prime$ be obtained from $G$ by switching so that one of the edges is in $E_1$.
Then $G^\prime  \in [G]_\Gamma$, but there is no $H' \in [H]_\Gamma$ such that $G' \to H'$.

%
%
%

\section{Switching with respect to an Abelian group}
\label{Abelian}

In this section we show that results of Brewster and Graves \cite{BG} about switching $m$-edge-coloured graphs with respect to a cyclic group extend to switching $(m, n)$-mixed graphs with respect to an Abelian group.

If the $(m, n)$-switching group $\Gamma$ is Abelian, then the order in which switches are applied at the vertices of the $(m, n)$-mixed graph $G$ does not matter.  
We may therefore assume that there is exactly one switch (which may be the with respect to the identity) 
applied at each vertex of $G$ and regard all switches as being applied simultaneously.
Thus the question of deciding whether $(m, n)$-mixed graphs $G$ and $H$ are $\Gamma$-switch equivalent is a finite question.
By contrast, when $\Gamma$ is non-Abelian it is not clear how many switches are required to transform $G$ into $H$, if it is possible to do so.
  
%

Let $\Gamma$ be an Abelian $(m, n)$-switching group, and $G$ be a $(m, n)$-mixed graph. 
The \emph{$\Gamma$-switching graph of $G$} is the $(m, n)$-mixed graph $P_\Gamma(G)$ with vertex set 
$V(P_\Gamma(G)) = V(G) \times \Gamma$, and edges and arcs defined as follows:
\begin{enumerate}
\item If $xy \in E_i(G)$ and $\gamma_1 = (\phi_1, \psi_1, \pi_1)$ and $\gamma_2 \in (\phi_2, \psi_2, \pi_2)$ are in $\Gamma$, then $(x, \gamma_1)(y, \gamma_2) \in E_{r}(P_\Gamma(G))$, where $r = \phi_1(i)\phi_2(i)$.
\item If $xy \in A_j(G)$ and $\gamma_1 = (\phi_1, \psi_1, \pi_1)$ and $\gamma_2 = (\phi_2, \psi_2, \pi_2)$ are in $\Gamma$, where
$\pi_1 = (p_{1,1}, p_{1,2}, \ldots, p_{1, n})$ and $\pi_2 = (p_{2,1}, p_{2,2}, \ldots, p_{2, n})$, then 
\begin{enumerate}
\item if $p_{1,j} = p_{2,j}$, then $(x, \gamma_1)(y, \gamma_2) \in E_{s}(P_\Gamma(G))$, where $s = \psi_1(j)\psi_2(j)$, and 
\item  if $p_{1,j} \neq p_{2,j}$, then $(y, \gamma_2)(x, \gamma_1) \in E_{s}(P_\Gamma(G))$, where $s = \psi_1(j)\psi_2(j)$.
\end{enumerate}
\end{enumerate}

\begin{observation}
Let $G$ and $H$ be $(m, n)$-mixed graphs and let $\Gamma$ be an Abelian $(m, n)$-switching group.
If $G$ is a subgraph  $H$, then $P_\Gamma(G)$ is a subgraph of $P_\Gamma(H)$. 
\end{observation}

We now show that, as in the cases of oriented graphs and the pushing operation, or $m$-edge coloured graphs and switching with respect to a cyclic group, the $\Gamma$-switching graph gives a representation of $[G]_\Gamma$. 

Let $\Gamma$ be an Abelian $(m, n)$-switching group.
Let $\mathcal{S} = (v_1, \gamma_1), (v_2, \gamma_2), \ldots, (v_t, \gamma_t)$ be a sequence of
switching pairs such that each vertex of the $(m, n)$-mixed graph $G$ is the first component of exactly one 
ordered pair in $\mathcal{S}$.
Then, there is a natural isomorphism of $G^{\mathcal{S}}$ to 
the subgraph of $P_\Gamma(G)$ induced by $\{(v_i, \gamma_i): 1 \leq i \leq t\}$.
In particular, the subgraph of $P_\Gamma(G)$ induced by 
$\{(x, e): x \in V(G)\}$ is isomorphic to $G$.  

Define a \emph{transversal subgraph} of $P_\Gamma(G)$ to be a subgraph of $P_\Gamma(G)$ induced by a transversal of the collection of the $|V(G)|$ sets $S_x = \{(x, \gamma): \gamma \in \Gamma\}, x \in V(G)\}$.

\begin{observation}
Let $G$ be a $(m, n)$-mixed graph.
Then $G' \in [G]_\Gamma$ if and only if $G'$ is a transversal subgraph of $P_\Gamma(G)$.
\end{observation}

The $\Gamma$-switching graph, $P_\Gamma(G)$,  transforms questions about $\Gamma$-switchable homomorphisms into questions about (ordinary) homomorphisms.

\begin{lemma}
Let $G$ be a $(m, n)$-mixed graph and let $\Gamma$ be an Abelian $(m, n)$-switching group.
Then  $P_\Gamma(G) \xrightarrow[\Gamma]{}  G$, and at least one such mapping is onto.
\label{P_GammaToG}
\end{lemma}

\begin{proof}
Note that,  for each $w \in V(G)$, the set of vertices $S_w = \{(w, \gamma): \gamma \in \Gamma\}$ is 
independent in $P_\Gamma(G)$.
At each vertex $(w, \gamma) \in V(P_\Gamma(G))$, switch with respect to $\gamma^{-1}$.

Suppose $xy \in E_i(G)$.
After switching, the edge $(x, \alpha)(y, \beta)$ belongs to  
$E_{\alpha\alpha^{-1}\beta\beta^{-1}(i)} = E_i$.
Thus, after switching as above, the mapping that sends all elements of $S_w$ to $w$ for each $w \in V$ is 
an onto $\Gamma$-switchable homomorphism $P_\Gamma(G) \xrightarrow[\Gamma]{}  G$.
\end{proof}

A \textit{$\Gamma$-switchable isomorphism} between $G$ and $H$ is an isomorphism $G'\cong H$, where $G'\in[G]_\Gamma$. If some such isomorphism exists we say \emph{$G$ is $\Gamma$-switchably isomorphic to $H$}, and write $G\cong_\Gamma H$. 

\begin{lemma}
Let $G,$ and $H$ be $(m, n)$-mixed graphs and let $\Gamma$ be an $(m, n)$-switching group. 
If  there are homomorphisms $G \xrightarrow[\Gamma]{} H$ and  $H \xrightarrow[\Gamma]{} G$, which are both onto (the relevant vertex set), 
then  $G \cong_\Gamma H$.
\label{SwitchIso}
\end{lemma}

\begin{proof}
Let $G^\prime \in [G]_\Gamma$ be such that $g: G^\prime \to H$ is onto, and
let $H^\prime \in [H]_\Gamma$ be such that $h: H^\prime \to G$ is onto.
Clearly $|V(G)| = |V(H)|$.
Thus, since $g$ and $h$ are onto, each edge or arc of $H$ is the image under $g$ of at most one edge or arc of $G^\prime$,
and similarly each edge or arc of $G$ is the image under $h$ of at most one edge or arc of $H^\prime$.
Therefore, the graphs $\mathit{underlying}(G)$ and $\mathit{underlying}(H)$ have the same number of edges.
It follows that both $g$ and $h$ are isomorphisms.
\end{proof}

\begin{theorem}
Let $G$ and $H$ be  $(m, n)$-mixed graphs and let $\Gamma$ be a $(m, n)$-switching group.
Then,
\begin{enumerate}
\item $G\cong_\Gamma H$ if and only if $P_\Gamma(G)\cong P_\Gamma(H)$;
\item $G\xrightarrow[\Gamma]{} H$ if and only if $P_\Gamma(G)\rightarrow P_\Gamma(H)$; 
\item $G\xrightarrow[\Gamma]{} H$ if and only if $G\rightarrow P_\Gamma(H)$.
\end{enumerate}
\label{BGThm} \label{MapSwitchGraphs} \label{HomCor}
\end{theorem}

\begin{proof}
We first prove statement 1.
Suppose $P_\Gamma(G) \cong P_\Gamma(H)$.
A switchable homomorphism of $G$ onto $H$ is obtained by composing
an embedding $G \to P_\Gamma(G)$,
an isomorphism $P_\Gamma(G) \to P_\Gamma(H)$, 
and an onto homomorphism 
$P_\Gamma(H) \to H$.
Similarly, there is a switchable homomorphism of $H$ onto $G$.
Lemma \ref{SwitchIso} now implies that $G \cong_\Gamma H$.

For the proof of the converse, suppose $G \cong_\Gamma H$.
Then there exists a sequence $\mathcal{S} = (v_1, \gamma_1), (v_2, \gamma_2), \ldots, (v_t, \gamma_t)$  of
switching pairs such that each vertex of the $(m, n)$-mixed graph $G$ is the first component of exactly one 
ordered pair in $\mathcal{S}$ and $G^\mathcal{S} \cong H$.

Define $f: V(P_\Gamma(G)) \to V(P_\Gamma(H))$
by $f((v_i, \alpha)) = (v_i, \alpha \gamma_i^{-1})$.
Then $f$ is a bijection.
Since $P_\Gamma(G)$ and $P_\Gamma(H)$ have the same number of edges,
it suffices to prove that $f$ is a homomorphism.

Suppose first that $(v_j, \alpha)(v_k, \beta) \in E_i(P_\Gamma(G))$.
Then $v_jv_k \in E_{\alpha^{-1} \beta^{-1}(i)}(G)$,
so that $v_jv_k \in E_{\gamma_i \gamma_j \alpha^{-1} \beta^{-1}(i)}(H)$.
Thus,
$f((v_j, \alpha)) f((v_k, \beta)) = 
(v_j, \alpha  \gamma_j^{-1})(v_k, \beta \gamma_k^{-1}) \in E_\ell(P_\Gamma(H))$,
where $\ell = \gamma_j \gamma_k$ $\alpha^{-1} \beta^{-1} \alpha \gamma_j^{-1} \beta  \gamma_k^{-1}(i) = i$.

The argument is similar (though notationally more cumbersome) when $(v_j, \alpha)(v_k, \beta) \in A_i(P_\Gamma(G))$.
This completes the proof of statement 1.

We now prove statement 2.
Suppose $G  \xrightarrow[\Gamma]{}  H$, and let  
$G^\prime \in [G]_\Gamma$ be such that 
there is a homomorphism $g: G^\prime \to H$.
Since $P_\Gamma(G) \cong P_\Gamma(G^\prime)$ by statement 1, 
it is enough to show $P_\Gamma(G^\prime) \to P_\Gamma(H)$.
Define $f: V(P_\Gamma(G^\prime)) \to V(P_\Gamma(H))$ by 
$f((w, \phi)) = (g(w), \phi)$.

Suppose $(u, \alpha)(v, \beta) \in E_j(P_\Gamma(G^\prime))$.
Then $uv \in E_i(G^\prime)$, where $i = \alpha^{-1}\beta^{-1}(j)$.
Since $g$ is a homomorphism, 
$g(u)g(v) \in E_i(H)$.
Thus,  $f((u, \alpha))f((v, \beta)) \in E_j(P_\Gamma(H))$.

The argument is the same when $(u, \alpha)(v, \beta) \in A_j(P_\Gamma(G^\prime))$.

For the proof of the converse, suppose $P_\Gamma(G) \to P_\Gamma(H)$.
Then, since $G \to P_\Gamma(G)$ and $P_\Gamma(H) \xrightarrow[\Gamma]{}  H$, we have 
$G \xrightarrow[\Gamma]{}  H$ by Proposition \ref{compose}.

We now prove statement 3.
Suppose $G \xrightarrow[\Gamma]{}  H$. Then, by statement 2, $P_\Gamma(G) \to P_\Gamma(H)$. 
The inclusion map gives $G \to P_\Gamma(G)$, so by Proposition \ref{compose}, we have $G \to P_\Gamma(H)$. 

For the proof of the converse, suppose $G \to P_\Gamma(H)$.
By Lemma \ref{P_GammaToG},  $P_\Gamma(H) \xrightarrow[\Gamma]{}  H$. 
The statement now follows from Observation \ref{Obs3.1} and Proposition \ref{compose}.
\end{proof}

\section{Switchable cores}

A $(m, n)$-mixed graph is a \emph{$\Gamma$-switchable core} if it admits no $\Gamma$-switchable homomorphism to a proper subgraph.  When $G$ is a $(1,0)$-mixed graph (i.e. a graph) or $(0, 1)$-mixed graph (i.e., a digraph), the definition coincides with the usual definition of the core of a graph.

As in the previous section, for Abelian $(m, n)$-switching groups $\Gamma$ we transform the problem of whether a graph is a $\Gamma$-switchable core into a problem that does not involve switching. The  $(m, n)$-mixed graph $P_\Gamma(G)$  is not in itself sufficient to accomplish this, but a special subgraph of $P_\Gamma(G)$ suffices.  The development below follows results implicit in \cite{BG} for $m$-edge-coloured graphs and switching with respect to cyclic groups.

For each vertex $x\in G$, define an equivalence relation $\equiv_x$ on $\Gamma$ by $\gamma_1 = (\phi_1, \psi_1, \pi_1) \equiv_{x}\gamma_2 = (\phi_2, \psi_2, \pi_2)$ if and only if 
\begin{enumerate}
\item $\phi_1(f) = \phi_2(f)$ for every edge $f$ incident with $x$,
\item $\psi_1(a) = \psi_2(a)$ for every arc $a$ incident with $x$, and
\item $\pi_1 = \pi_2$.
\end{enumerate}

Let $G$ be an $(m, n)$-mixed graph.
The $(m, n)$-mixed graph $S_\Gamma(G)$ is obtained from $P_\Gamma(G)$ by identifying vertices $(x, \gamma_1)$ and $(x, \gamma_2)$ whenever $\gamma_1 \equiv_x \gamma_2$ (note: necessarily the same $x$).
Then $S_\Gamma(G)$ is a transversal subgraph of $P_\Gamma(G)$.
Further, the mapping just defined is a homomorphism.
A $(m, n)$-mixed graph $H$ is a \emph{retract} of a $(m, n)$-mixed graph $G$ if it is an induced subgraph of $G$ and there is a homomorphism $G \to H$ that maps each vertex of $H$ to itself.

\begin{theorem}
Let $G$ be a $(m, n)$-mixed graph, and $\Gamma$ be an Abelian $(m, n)$-switching group.  Then $S_\Gamma(G)$ is a retract of $P_\Gamma(G)$.  \label{S_Gamma}
\end{theorem}

\begin{proof}
Let $r: V(P_\Gamma(G)) \to V(S_\Gamma(G))$ be defined by $r((x, \alpha)) = [\alpha]_{\equiv_x}$.
If $(x, \alpha)(y, \beta) \in E_i(P_\Gamma(G))$, then by the definitions of $S_\Gamma(G)$ and the relations $\equiv_x$ we have $[\alpha]_{\equiv_x}[\beta]_{\equiv_y} \in E_i(S_\Gamma(G))$.  The argument is the same when
$(x, \alpha)(y, \beta) \in A_i(P_\Gamma(G))$.
\end{proof}

The results of the previous section can be seen to hold using $S_\Gamma$ in place of $P_\Gamma$.
We have chosen not to do so because the latter $(m, n)$-mixed graph is conceptually easier to work with.
The $(m, n)$-mixed graph $S_\Gamma$ is useful in determining the complexity of deciding whether a 
given $(m, n)$-mixed graph is a $\Gamma$-switchable core.

%

Welzl \cite{Welzl} and Fellner \cite{Fellner} independently proved that every graph $G$ has a unique induced subgraph $H$ which is a core and for which there is a homomorphism $G \to H$, where uniqueness is up to isomorphism.
The same result holds for $(m, n)$-mixed graphs with the identical proof.  

\begin{theorem} {\rm \cite{Fellner, Welzl}}
Up to isomorphism, every $(m, n)$-mixed graph $G$ has a unique induced subgraph $H$ which is a core and for which there is an onto homomorphism $G \to H$.  
\label{CoreThm}
\end{theorem}

This uniquely defined subgraph $H$ in the previous theorem is called \emph{the core of $G$}.

We now show that Theorem \ref{CoreThm} extends to $\Gamma$-switchable cores.  The proof is essentially a direct translation of Welzl's proof to $\Gamma$-switchable homomorphisms.

\begin{theorem}
Let $G$ be a $(m, n)$-mixed graph and let $\Gamma$ be a $(m, n)$-switching group.
Up to $\Gamma$-switchable isomorphism,  $G$ has a unique induced subgraph $H$ which is a core and for which there is an onto 
$\Gamma$-switchable homomorphism $G \xrightarrow[\Gamma]{} H$.  
\label{SwitchCoreThm}
\end{theorem}

{\bf Proof}.
Suppose $H_1$ and $H_2$ are both $\Gamma$-switchable cores of $G$.  We will show that $H_1 \cong_\Gamma H_2$.  Since $H_1$ is a $\Gamma$-switchable core of $G$, and $H_2$ is an induced subgraph of $G$, there 
is a homomorphism $H_2 \to G$ and a
$\Gamma$-switchable homomorphism $G \xrightarrow[\Gamma]{} H_1$.
Thus,  there exists a $\Gamma$-switchable homomorphism $f: H_2 \xrightarrow[\Gamma]{} H_1$ and,
similarly, a $\Gamma$-switchable homomorphism $g: H_1 \xrightarrow[\Gamma]{} H_2$. 
Composing $f$ and $g$ gives $\Gamma$-switchable homomorphisms
$H_2 \xrightarrow[\Gamma]{} H_2$ and $H_1 \xrightarrow[\Gamma]{} H_1$, both of which must be onto by the definition of a $\Gamma$-switchable core.
Therefore $f$ and $g$ are both onto.
The result now follows from Lemma \ref{SwitchIso}.
$\Box$

This uniquely defined subgraph $H$ in the previous theorem is called \emph{the $\Gamma$-switchable core of $G$}.

\begin{theorem}
Let $G$ be a $(m, n)$-mixed graph and let $\Gamma$ be an Abelian $(m, n)$-switching group.
Then $G$ is a $\Gamma$-switchable core if and only if $S_\Gamma(G)$ is a core.
\end{theorem}

{\bf Proof}.
Suppose $G$ is not a $\Gamma$-switchable core.
Then there exists a proper subgraph $H$ of $G$ such that $G \xrightarrow[\Gamma]{} H$.
Therefore $S_\Gamma(G) \to S_\Gamma(H)$ (as $S_\gamma(G) \to P_\Gamma(G) \to P_\Gamma(H) \to S_\Gamma(H)$).
Since $S_\Gamma(H)$ is a proper subgraph of $S_\Gamma(G)$, it follows that 
$S_\Gamma(G)$ is not a core.

Now suppose $S_\Gamma(G)$ is not a core.
Then it admits a homomorphism, $f$, to a proper subgraph.
Note that the vertices of $S_\Gamma(G)$ can be regarded as ordered pairs $(v, \gamma)$,
where $v \in V(G)$ and $\gamma \in \Gamma$, rather than as equivalence classes of such
ordered pairs; for notational convenience we write $(v, \gamma)$ instead of $[(v, \pi)]_{\equiv_v}$.

Let $S'$ be the $(m, n)$-mixed graph obtained from $S_\Gamma(G)$ by
switching at each vertex $(v, \gamma)$ with respect to $\gamma^{-1}$.
By definition of $S_\Gamma(G)$, and as in the proof of Theorem \ref{MapSwitchGraphs}, 
if $(x, \pi)(y, \rho) \in E_i(S')$ 
then $xy \in E_i(G)$.  Similarly, if
$(x, \pi)(y, \rho) \in A_j(S')$, then $xy \in A_j(G)$.
Therefore,  the function $p: V(S') \to V(G)$  that sends
$(v, \gamma)$ to $v$ for all $(v, \gamma) \in V(S')$ is a homomorphism,
and hence 
a $\Gamma$-switchable homomorphism $S_\Gamma(G) \xrightarrow[\Gamma]{} G$.

Suppose there exists a vertex $x \in V(G)$ such that no vertex in
the set $R_x = \{(x, \pi): \pi \in \Gamma\}$ belongs to the range of $f$. 
Let $h: G \to S_\Gamma(G)$ be the homomorphism defined
by $h(v) = (v, e)$ for all $v \in V(G)$.
Let $G'$ be the subgraph of $S_\Gamma(G)$ induced by $\{(v, e): v \in V(G)\}$.
Then $G' \cong G$. 
The restriction of the function $p \circ f \circ h$ to $V(G')$ is
a $\Gamma$-switchable homomorphism of $G$ to a proper subgraph of $G$, 
it follows that $G$ is not a $\Gamma$-switchable core.

The previous implies that it suffices to show that
there is a 
$\Gamma$-switchable homomorphism 
$g: S_\Gamma(G) \xrightarrow[\Gamma]{} S_\Gamma(G)$
for which there  exists a vertex $x \in V(G)$ such that no vertex in
$R_x$ belongs to the range of $f$.  We now show such a homomorphism exists.

Since $f$ is a homomorphism of $S_\Gamma(G)$ to a proper subgraph, there
exists a vertex $(v, \gamma)$ which is not in the range of $f$.
We claim that $f((v, \gamma)) \not \in R_v$.
By definition of $S_\Gamma(G)$ for each vertex $(v, \pi) \in R_v \setminus \{(v, \gamma)\}$
there exists a vertex $(w, e)$ so that the 
edge or arc joining $(w, e)$ and $(v, \gamma)$ has a different colour or orientation
(if appropriate) from the edge or arc joining $(w, e)$ and $(v, \pi)$.
Since $f$ is a homomorphism, it can not map $(v, \gamma)$ to  $(v, \pi)$ for any $w \in V(G')$.
The claim now follows.

Let $S''$ be obtained from $S_\Gamma(G)$ by switching at each
vertex $(v, \pi) \in R_v \setminus \{(v, \gamma)\}$ with respect to $\pi^{-1}\gamma$.
Then, the function $g:V(S'') \to V(S'')$ defined by 
$$g((x, \rho)) = \begin{cases}
f((x, \rho)) & (x, \rho) \not\in R_v\\
f((v, \gamma)) & (x, \rho) \in R_v
\end{cases}$$
is a homomorphism, and hence a 
$\Gamma$-switchable homomorphism $S_\Gamma(G) \xrightarrow[\Gamma]{} S_\Gamma(G)$.
Since no vertex in $R_v$ is in the range of $g$, the proof is complete.
$\Box$

Hell and Ne\v{s}et\v{r}il \cite{HN} proved that it is NP-complete to decide that a given graph or digraph is not a core.  
Applying their result to $(m, n)$-mixed graphs with only edges of one colour, or only arcs of one colour,  has the following consequence.

\begin{corollary}
The problem of deciding whether a given $(m, n)$-edge-coloured graph is not a core is NP-complete.
\end{corollary}

We now determine a $\Gamma$-switchable version of the above corollary. 
For an arbitrary $(m, n)$-switching group there is no guarantee that the list of the switching sequence 
used to transform a given $(m, n)$-mixed graph $G$ to that it has a homomorphism to a proper 
subgraph has length polynomial in $|V|$, so we can prove only NP-hardness.


\begin{theorem}
Let $G$ be a $(m, n)$-mixed graph and get $\Gamma$ be a $(m, n)$-switching group.
The problem of deciding whether $G$ is not a $\Gamma$-switchable core is NP-hard.
\end{theorem}
{\bf Proof.}
Suppose first that $m > 0$.  The transformation is from the problem of deciding whether a given graph $H$ is not a core.    The transformed instance is the $(m, n)$-mixed graph $H'$ whose underlying graph is $G$ and all of whose edges are the same colour.  It can clearly be accomplished in polynomial time. It follows from 
Observation \ref{UnderlyingHom}
that $H'$ is not a $\Gamma$-switchable core if and only if $H$ is not a core.

The proof is the same when $m = 0$ and $n > 0$.  The result now follows.  
$\Box$

When the $(m, n)$-switching group $\Gamma$ is Abelian, the switching sequence 
used to transform a given $(m, n)$-mixed graph $G$ to that it has a homomorphism to a proper 
subgraph has length polynomial in $|V|$.

\begin{corollary}
Let $\Gamma$ be an Abelian $(m, n)$-switching group.
The problem of deciding whether a given $(m, n)$-mixed graph is not a $\Gamma$-switchable core is NP-complete.
\end{corollary}

\section{Switchable colourings}

We will preset a dichotomy theorem for the problem of deciding whether a given
$(m, n)$-mixed graph 
is $k$-colourable, where $k$ is a positive integer.

Let $G$ be a simple graph.
For any $k$-colouring of $G$, identifying the vertices assigned the same colour results
in a simple graph $H$ such that there is a homomorphism $G \to H$.
Thus a $k$-colouring of $G$ can be defined as a homomorphism of $G$ to a (i.e. some)
simple graph on $k$ vertices.  
Since the existence of a homomorphism of $G$ to $H$ implies the existence of a homomorphism of 
$G$ to any supergraph of $H$, an equivalent definition is that 
a  $k$-colouring of $G$ is a homomorphism $G \to K_k$.

Now let $G$ be a $(m, n)$-mixed graph, and $k$ be a positive integer.  
(Recall that $\mathit{underlying}(G)$ is assumed to be simple.)
A \emph{$k$-colouring} of $G$ is a homomorphism of $G$ to a  $(m, n)$-mixed
graph $H$ on $k$ vertices.   For the same reason as above, the underlying graph of
$H$ can be assumed to be $K_k$.

When $G$ is a  $(1, 0)$-mixed graph, this notion of colouring coincides with the usual notion of graph colouring.
When $G$ is a $(0, 1)$-mixed graph, this notion of colouring coincides with oriented colouring (e.g. see \cite{Sopena})

If $\Gamma$ is a $(m, n)$-switching group, then a 
\emph{$\Gamma$-switchable $k$-colouring of $G$} is a $k$-colouring of some $G' \in [G]_\gamma$.

When $G$ is a  $(m, 0)$-mixed graph and $\Gamma$ is a cyclic group, then $\Gamma$-switchable $k$-colouring coincides with the colourings studied in \cite{BFHN, BG}.
When $G$ is a   $(0, 1)$-mixed graph and the group $S_2$ reverses the orientation of arcs incident with a vertex, then $S_2$-switchable $k$-colouring coincides with a pushable $k$-colouring \cite{KM}.
Dichotomy theorems for $S_2$-switchable $k$-colouring of  $(2, 0)$-mixed graphs, and pushable $k$-colouring of 
 $(0, 1)$-mixed graphs are known \cite{BFHN, KM}.  In both cases the problem of deciding whether $G$ has a 
$k$-colouring is solvable in polynomial time if $G$ has a 2-colouring, and is NP-complete if $G$ has no 2-colouring.

\begin{theorem}
Let $\Gamma$ be a  Abelian $(m, n)$-switching group, and let $k$ be a positive integer.
The problem of deciding whether a given  $(m, n)$-mixed graph $G$, has a $\Gamma$-switchable $k$-colouring
is solvable in polynomial time if $k \leq 2$.
If $m \geq 1$ and  $k \geq 3$, then the problem of deciding whether a given  $(m, n)$-mixed graph $G$, has a $\Gamma$-switchable $k$-colouring
is NP-complete.
\end{theorem}

{\bf Proof}.
We first prove NP-completeness.  
Since $\Gamma$ is Abelian, the problem is clearly in NP.
The transformation is from $k$-colouring.
Let the simple graph $G$ be an instance of $k$-colouring.
The transformed instance of $\Gamma$-switchable $k$-colouring is the  $(m, n)$-mixed graph $G'$ with all edges of the same colour (the colour chosen is unimportant).
Since a $k$-colouring of $G$ corresponds to a $k$-colouring of $G'$ -- no switching needed -- and a $\Gamma$-switchable $k$-colouring of $G'$ corresponds to a $k$-colouring of $\mathit{underlying}(G') = G$, the result follows.

Since a  $(m, n)$-mixed graph $G$ has a $\Gamma$-switchable 1-colouring if and only if it has neither edges nor arcs, it suffices to show that the problem of deciding whether a given  $(m, n)$-mixed graph $G$ has a $\Gamma$-switchable $2$-colouring is solvable in polynomial time.

If $G$ has edges and arcs, then $G$ is not 2-colourable. 
Every $G' \in [G]_\Gamma$ has both edges and arcs and no $(m,n)$ mixed graph with both
edges and arcs is 2-colourable.  Hence either $E(G) = \emptyset$ or $A(G) = \emptyset$.

If $\Gamma$ does not act transitively on the edge or arc colours and $G$ has edges or arcs from two different orbits, 
then $G$ is not 2-colourable.
Every $G' \in [G]_\Gamma$ has edges or arcs from two different orbits, and 
no $(m, n)$-mixed graph with edges or arcs of different colours is 2-colourable.
We may therefore assume that $\Gamma$ acts transitively on the set of edge or arc colours (as appropriate).

Suppose first that $A(G) = \emptyset$.
By Theorem \ref{HomCor}, $G$ has a $\Gamma$-switchable $2$-colouring if and only if $G$ has a homomorphism to 
$P_\Gamma(H)$, where $\mathit{underlying}(H) = K_2$ and the colour the edge of $H$ is unimportant because $\Gamma$ acts transitively.
Since an Abelian group that acts transitively also acts regularly, each vertex of $P_\Gamma(H)$ is incident with exactly
one edge of each colour.  It is shown in \cite{RickThesis} that the existence of a homomorphism to a $(m, 0)$-graph 
with this property can be decided in polynomial time.

When $E(G) = \emptyset$ the argument is the same except that the vertices of $H$ are joined by an arc of some colour.
The fact that $\Gamma$ may reverse the orientation of arcs of some colours does not matter.
The essential property is that each vertex of $P_\Gamma(H)$ is incident with exactly
one arc of each colour. 

This completes the proof.
$\Box$

The above theorem does not cover the cases where $k \geq 3$ and $m = 0$.
It therefore provides a dichotomy for $\Gamma$-switchable $k$-colourings of  $(m, 0)$-mixed graphs ($m$ edge-coloured graphs) and more, but not for all simple $(m, n)$ mixed graphs.
The situation in the remaining cases seems to be complicated by how $\Gamma$ 
possibly alters the orientation of arcs.
Completing the dichotomy for $\Gamma$-switchable $k$-colourings of simple $(m, n)$-mixed graphs is left
as a project for future work.

\end{document}